\documentclass{amsart}
\usepackage{amsfonts,amsmath,amsthm,amscd,amssymb,latexsym, mathrsfs}

\newtheorem{theorem}{Theorem}[section]
\newtheorem*{m-theorem1}{Theorem 5.7}
\newtheorem*{m-theorem2}{Theorem 5.14}
\newtheorem*{m-theorem3}{Corollary 5.15}
\newtheorem*{m-theorem4}{Corollary 5.11}
\newtheorem{lemma}[theorem]{Lemma}
\newtheorem{corollary}[theorem]{Corollary}
\newtheorem{proposition}[theorem]{Proposition}
\newtheorem{definition}[theorem]{Definition}

\newtheorem{remark}[theorem]{Remark}

\DeclareMathOperator{\rk}{rk}
\DeclareMathOperator{\gk}{\mathscr{GK}}
\DeclareMathOperator{\s}{\ast}
\DeclareMathOperator{\res}{res}
\DeclareMathOperator{\supp}{Supp}

\DeclareMathOperator{\hm}{Hom}
\DeclareMathOperator{\codim}{co-dim}
\DeclareMathOperator{\aut}{Aut}

\begin{document}
\date{}


\title{The Krull and global dimension of the tensor product of $n$-dimensional quantum tori}

\maketitle

\section{Introduction}
\label{intr}

Let $F$ be a field. The associative $F$-algebra generated by $X_1, \cdots X_n$ together with their inverses satisfying the relations
\begin{equation}
\label{q_c_r}
X_iX_j = \lambda_{ij}X_jX_i \ \ \ \ \ \ \lambda_{ij} \in F\setminus \{0\}  \ \ \ \forall  i, j \in \{1, \cdots, n \}
\end{equation} 
 is known as the quantum torus as it is a quantum deformation of the coordinate algebra of the torus. It plays an important role in non-commutative geometry, quantum groups and group representation theory.  It is  also  interesting in its own right being an example of a noncommutative polynomial algebra.   

The case  $n = 2$ is of special interest. 
Here the defining relation is   $XY = qYX$     
where $q$ is a nonzero scalar in $F$ and we denote the resulting algebra as $B_q$. 
It was shown in \cite{Ja} and \cite{Lo} that when $q$ is not a root of unity, $B_q$ resembles the first Weyl algebra $A_1(k)$ where $k$ is a field of characteric zero.
In \cite{Ja} iterated tensor products of the type 
\[ B_{q_1} \otimes_F \cdots \otimes_F B_{q_k} \] 
were also considered.
 
Returning to the general case we note that the \emph{multiparameters} $\lambda_{ij}$ satisfy the conditions
\[ \lambda_{ii} = 1 = \lambda_{ij}\lambda_{ji}  \ \ \ \ \ \forall i,j \in \{1, \cdots, n\}. \] 

Let $\Lambda := (\lambda_{ij})$. We denote the quantum torus defined in (\ref{q_c_r}) by $P(\Lambda)$.  
We note that the algebras $P(\Lambda)$ are precisely the twisted group algberas $F \s A$ of finitely generated free abelian groups $A$ over $F$. 
In \cite{MP},  J.C. McConnell and J.J. Pettit invesitaged the ring theoretic dimensions of the algebras $P(\Lambda)$. It was shown in this paper that the Krull and global dimensions of these algebras coincide and if $d$ denotes this common value then  $1 \le d \le n$. Henceforth by the dimension of $P(\Lambda)$ we shall mean any of these dimensions.

A criterion (\cite[Corollary 3.8]{MP}) for the exact value of $d$ in terms of certain partial localizations of $P(\Lambda)$ was also given in the same paper. 
It was moreover conjectured that if $F \s A$ is  the underlying twisted group algebra of $P(\Lambda)$, then $d$ is the supremum of the ranks of the subgroups $B \le A$ for which the subalgebra $F \s B$ is commutative. Using the criterion for $d$ mentioned above and a geometric invariant for $F \s A$-modules introduced and studied in \cite{BG1}--\cite{BG3}, this conjecture was shown to be true by C.J.B. Brookes in \cite{Br2}.

 Given quantum tori $F \s A_1$ and $F \s A_2$ we can take their tensor product over $F$ (this turns out to be a twisted group algebra $F \s (A_1 \times A_2)$). The question then arises as to how the dimension of this tensor product is related to the dimensions of the (tensor) factors $F \s A_1$ and $F \s A_2$.  

From the basic properties of tensor product of algebras and the result of Brookes it immediately follows that 
\begin{equation}\label{ineq_0}
 {\dim(F \s A_1 \otimes_F F \s A_2 ) \ge \dim(F \s A_1) + \dim(F \s A_2),}
\end{equation}
where we denote by $\dim(F \s A)$ either the Krull or the global dimension of the algebra $F \s A$. In other words the dimension is in general \emph{super-additive} with respect to tensoring.

It is easy to find examples where equality does not hold in (\ref{ineq_0}). 
As the tensor product $F \s A_1 \otimes_F F \s A_2$ is a twisted group algebra $F \s (A_1 \times A_2)$,
its dimension can be at most $\rk(A_1) + \rk(A_2)$.  

Our first result gives the following upper bound for the dimension of the tensor product assuming that 
the dimension of each algebra $F \s A_i$ is not the maximum possible.

\begin{m-theorem1}
Given algebras $F \s A_1$ and $F \s A_2$ suppose that $\dim(F \s A_i) < \rk(A_i)$ for $i \in \{1,2\}$. 
Let $d := \dim(F \s A_1 \otimes_F F \s A_2)$.
Then
\begin{equation}
\label{ineq_1}
d \le \min\{ \dim(F \s A_1) + \rk(A_2), \dim(F \s A_2) + \rk(A_1)   \} - 1. 
\end{equation}  
\end{m-theorem1}

It may be of interest to know when is the dimension of a tensor product \emph{additive} with respect to tensoring. In this connection, the following corollary can be deduced from Theorem 5.7.

\begin{m-theorem4}
Given twisted group algebras $F \s A_1$ and $F \s A_2$, 
if $\dim(F \s A_i) = \rk(A_i) - 1$ for some $i \in \{1,2 \}$ then 
\[ \dim(F \s A_1 \otimes_F F \s A_2) = \dim(F \s A_1) + \dim(F \s A_2) \]
\end{m-theorem4}

The upper bound of (\ref{ineq_1}) is  the best possible. For example, let
$ \Lambda = (\lambda_{ij})$ be a multiplicatively antisymmetric $n \times n$ matrix  such that the 
subset $\{ \lambda_{ij} \mid 1 \le i < j \le n \}$ of $F$ is multiplicatively independent. By \cite[Corollary 3.10]{MP}, 
$\dim(P_\Lambda) = 1$.    
Let $\Lambda'$ be the transpose of $\Lambda$. Then $\Lambda'$ is also multiplicatively antisymmetric and 
$\dim(P_{\Lambda'}) = 1$ by the same result. 
 
Suppose that $P_{\Lambda'}$ is generated over $F$ by the indeterminates $X_i'$, ($i \in \{1, \cdots, n\}$) together with their inverses. Our choice of the defining multiparameters ensures that the monomials  
\[ \{ X_i \otimes X_i'\mid 1 \le i \le n \} \]   
commute mutually in $P_{\Lambda} \otimes_F P_{\Lambda'}$. But then $\dim(P_{\Lambda} \otimes_F P_{\Lambda'})\ge n$ in view of Brookes' theorem (Theorem \ref{dim_sup}). 
     
Writing $P_{\Lambda}$ as $F \s A$ and similarly $P_{\Lambda'}$ as $A'$, where both $A$ and $A'$ have rank $n$, we have in view of (\ref{ineq_1}) that

\begin{align*}
\dim( P_\Lambda \otimes_F  P_{\Lambda'}) &\le \min\{ \dim(F \s A) + \rk(A'), \dim(F \s A') + \rk(A) \} - 1 \\ 
				       &=  \min\{ 1 + n , 1 + n \} - 1 \\ 
					   &=  n.
\end{align*}

This shows that the upper bound in (\ref{ineq_1}) is the best possible. 
There are also cases in which the upper bound in (\ref{ineq_1}) is not attained by the dimension of the tensor product $F \s A_1 \otimes_F F \s A_2$. Theorem \ref{dim_<_ub} gives sufficient conditions for this.

\begin{m-theorem2}
Let $F \s A_1$ and $F \s A_2$ be twisted group algebras such that
\begin{enumerate}
\item [(i)] $\dim(F \s A_i) \ge 2$,
\item [(ii)] $\rk(A_i) - \dim(F \s A_i) \ge 2$,
\item [(iii)] $F \s A_i$ has center $F$.

\end{enumerate}
Let $d : = \dim(F \s A_1 \otimes F \s A_2)$. Then  
\begin{equation*}
d < \min\{\dim(F \s A_1) + \rk(A_2), \dim(F \s A_2) + \rk(A_1) \} - 1. 
\end{equation*}
\end{m-theorem2}
It is convenient to define codimension of the algebra $F \s A$ as 
\[ \codim(F \s A) = \rk(A) - \dim(F \s A). \] 
We can deduce the following corollary from Theorem 5.14 which gives further cases where the dimension 
is additive with respect to tensor products.

\begin{m-theorem3}
Suppose that the algebras  $F \s A_1$ and $F \s A_2$ satisfy the following conditions
\begin{itemize}
\item[(i)] $\dim(F \s A_1),\dim(F \s A_2) \ge 2$,
\item[(ii)] $\codim(F \s A_1) \ge \codim(F \s A_2) = 2$, 
\item[(iii)] $F \s A_i$ has center $F$.  
\end{itemize}
Then \[ \dim(F \s A_1 \otimes F \s A_2) = \dim(F \s A_1) + \dim(F \s A_2) \] 
\end{m-theorem3}

This article is organised as follows. We begin by reviewing twisted group algebras $F \s A$ and their localizations in Section \ref{t_g_a}. 
In Section \ref{b_str_mod} some facts on $F \s A$-modules are recalled. Section 4 gives
 an exposition of a geometric invariant of Brookes and Groves for $F \s A$-modules which we use in our 
invesgtigations The problem of the dimension of tensor products is discussed in Section 5. 
    
\section{Twisted group algebras and crossed products}
\label{t_g_a}

Let $F^* := F \setminus \{0\}$. Let $A$ denote a finitely generated free abelian group. 
We denote by $\rk(A)$ the rank of $A$. 
An $F$-algebra $\mathcal A$ is a \emph{twisted group algebra} $F \s A$ of $A$ over $F$ 
if $\mathcal A$ has a copy $\overline A := \{ \bar a : a \in A \}$ of $A$ which is an $F$-basis so that the multiplication in 
$\mathcal A$ satisfies  
\begin{equation}
\label{tws}
\bar {a}_1 \bar {a}_2 = \tau(a_1, a_2)\overline{a_1a_2} \ \ \ \ \ \forall a_1, a_2 \in A,
\end{equation} 
where $\tau : A \times A \rightarrow F^*$ is a function satisfying
\[ \tau(a_1, a_2)\tau(a_1a_2, a_3) = \tau(a_2, a_3)\tau(a_1, a_2a_3) \ \ \ \ \ \forall a_1, a_2, a_3 \in A.  \]
For $a_1, a_2 \in A$, it easily follows from (\ref{tws}) that the group-theoretic commutator $[\bar a_1, \bar a_2] \in F^*$. 
The following identities thus follow from the
basic properties of commutators (see, for example, \cite[Section 5.1.5]{Ro}):
\begin{align}\label{com_id_1}
[\bar {a}_1 \bar {a}_2, \bar {a}_3] &= [\bar {a}_1, \bar {a}_3] [\bar {a}_2, \bar {a}_3],\\ \label{com_id_2}
[\bar {a}_1,  \bar {a}_2 \bar {a}_3] &= [\bar {a}_1, \bar {a}_2] [\bar {a}_1, \bar {a}_3], \\ \label{com_id_3}
[\bar {a}_1, \bar {a}_2^{-1}] &= [\bar {a}_1, \bar {a}_2]^{-1}, \\ \label{com_id_4} 
[\bar {a}_1^{-1}, \bar {a}_2] &= [\bar {a}_1, \bar {a}_2]^{-1}
\end{align}
For a subset $X$ of $A$, we define $\overline{X} = \{\bar x : x \in X \}$. 
If $X, Y \subset A$, we set 
\[ [\overline X, \overline Y] = \langle [\bar x, \bar y] : x \in X, y \in Y \rangle. \]
Thus $[\overline X, \overline Y]$ is a subgroup of $F^*$. If $\alpha \in F \s A$, we may express $\alpha = \sum_{a \in A} \lambda_a \bar a$, where $\lambda_a \in F$ and  $\lambda_a = 0$ for ``almost all" $a \in A$.
We define the support of $\alpha$ (in $A$) as  
\[ \supp(\alpha) = \{ a \in A : \lambda_a \ne 0 \}.\]
Note that for a subgroup $B$ of $A$, the subalgebra generated by $\overline B \subset F \s A$ is a twisted group algebra $F \s B$. 
It was shown in \cite[Proposition 1.3]{MP} that an algebra $F \s A$ is simple if and only if it has center $F$. 
\begin{proposition}
\label{c_a_fis}
An algebra $F \s A$ has center exactly $F$ if and only if for each subgroup $A_1 < A$ with finite index 
$F \s A_1$ has center $F$.
\end{proposition}
\begin{proof}
Suppose that $F \s A$ has center $F$. Let $A_1 \le A$ be a subgroup such that $l : = [A : A_1] < \infty$. 
We claim that $F \s A_1$ also has center $F$. Using \cite[Proposition 1.3]{MP}, we may assume that $\bar {a}_1$ is central in $F \s A_1$ for $1 \ne a_1 \in A_1$. For any $a \in A$, (\ref{com_id_1}) and (\ref{com_id_2}) yield:
\[ [ \bar {a}_1^l, \bar a ] = [ \bar {a}_1, \bar a ]^l =  [ \bar {a}_1, \bar a^l ] = 1, \] 
where the last equality holds since $a^l \in A_1$. 
Since $A$ is torsion-free by definition, $1 \ne a_1^l$. Thus $\bar {a}_1^l$ is a nonscalar central element of $F \s A$. 
The converse is clear. 
\end{proof}

It is well known (see, for example, \cite[Lemma 37.8]{Pa2}) that for each subgroup $B \le A$, 
\[ \mathcal S_B : = F \s B \setminus \{0\} \] 
is an Ore subset in $F \s A$. 
Thus $F \s A$ has a localization $(F \s A)S_B^{-1}$ at $S_B$ which has the structure of a \emph{crossed product} 
$D_B \s A/B$ of $A/B$ over the quotient division ring $D_B$ of $F \s B$.

A crossed product $D \s C$ of an abelian group $C$ over a division ring $D$ is an associative ring which has a 
a copy $\overline C : = \{\bar c \mid c \in C \}$ of $C$ as $D$-basis with 
the multiplication satisfying 
\begin{align*}
\bar {c}_1 \bar {c}_2 &= \tau (c_1, c_2)\overline {c_1c_2} \ \ \ \ \ \  \forall c_1, c_2 \in C, \\  
\bar {c} d &= \sigma_c(d) \bar c, \ \ \ \ \ \ \ \ \ \ \ \ \ \forall c \in C, \ \ \forall  d \in D, 
\end{align*}
for a suitable function  $\tau: C \times C \rightarrow D^*$ and automorphisms $\sigma_c \in \aut(D)$. 
We note that twisted group algbras are special types of crossed products arising when the division ring $D$ lies in the center. We refer to \cite{Pa2} for further details on crossed products.

\section{GK dimension and critical  modules}\label{b_str_mod}

All modules that we shall consider shall be right modules.
In our investigations we shall have to consider modules over suitable  localizations of the algebras $F \s A$ as just discussed.
Thus we shall now briefly discuss the basic module theory of crossed products $D \s A$, where $D$ is a division ring.

Let $M$ be a finitely generated $D \s A$-module and $B \le A$ be a subgroup.  
It is known that $\mathcal S_B : = D \s B \setminus \{0\}$ is an Ore subset (e.g.,\cite[Lemmma 37.8]{Pa2}) 
and it easily follows from this that
 \[ T_B(M) : = \{ m \in M \mid  m\beta = 0 \ \mbox{for some}\ \beta \in \mathcal S_B \}, \]   
is an $F \s A$-submodule of $M$ known as the $F \s B$-torsion submodule of $M$. 
We shall say that $M$ is $F \s B$-torsion-free if $T_B(M) = 0$ and 
$F \s B$-torsion if $T_B(M) = M$. 

In \cite{BG2}, a dimension for $D \s A$-modules was introduced and was shown to coincide with the standard Gelfand-Kirillov dimension (GK dimension) measured over $D$. The next proposition characterizes the GK dimension for finitely generated $D \s A$-modules.

\begin{proposition}[Brookes and Groves]
\label{GKdim_ch} 
Let $M$ be a finitely generated $D \s A$-module. Then $\gk(M)$ equals the supremum of the ranks of subgroups 
$B$ of $A$ such that $M$ is not torsion as $D \s B$-module. Furthermore, let $a_1, \cdots, a_n$ 
freely generate $A$ and define 
\[ \mathcal F = \{ \langle X \rangle : X \subset \{a_1, \cdots, a_n \} \} \]
with the convention that $\langle \emptyset \rangle = \langle 1 \rangle $. 
Then $\gk(M)$ is simply the supremum of the ranks of subgroups $B$ 
in $\mathcal F$ such that $M$ is not $D \s B$-torsion.   
\end{proposition}

\begin{remark}\label{loc_rmk}
Let us note some consequences. Let $\{x_1, \cdots, x_n \}$ be a basis of $A$. 
Let $M$ be a finitely generated $F \s A$-module with $\gk(M) = r$, where $0 < r < \rk(A)$.
By the preceding proposition, there must be a nonempty subset 
$\{ i_1, i_2, \cdots, i_r \} \subset \{1, 2, \cdots, n \}$ 
so that $M$ is not 
$F \s \langle  x_{i_1}, x_{i_2}, \cdots, x_{i_r} \rangle$-torsion. 
Let $\{j_1, \cdots, j_{n - r} \}$ be the (set) complement to $\{i_1, i_2, \cdots, i_r \}$ in $\{1, \cdots, n \}$. 
and set $\mathcal S : = F \s \langle  x_{i_1}, x_{i_2}, \cdots, x_{i_r} \rangle \setminus \{0\}$. 
As already noted, $S$ is an Ore subset in $F \s A$ and the localization $(F \s A)S^{-1}$ is a crossed product 
$D \s \langle  x_{j_1}, \cdots, x_{j_{n - r}} \rangle$, where $D$ is the quotient division ring of $F \s \langle  x_{i_1}, x_{i_2}, \cdots, x_{i_r} \rangle$. For convenience, we can assume that $\{i_1, i_2, \cdots, i_r\} = \{1, 2, \cdots, r\}$ 
\end{remark}

Following the notation in \cite{MP}, we shall write $F( i_1, i_2, \cdots, i_r)$ 
for $D$ and denote $D \s \langle  x_{j_1}, \cdots, x_{j_{n - r}}  \rangle$ as 
\[ F( x_{i_1}, x_{i_2}, \cdots, x_{i_r})[ x_{j_1}, \cdots, x_{j_{n - r}} ].\] 

\begin{lemma}\label{fin_D_dim}
With the above notation, $M \mathcal S^{-1}$ is a nonzero $D \s  \langle x_{j_i}, \cdots, x_{j_{n - r}}  \rangle$-module 
that is finite dimensional as $D$-vector space. 
\end{lemma}

\begin{proof}
By  \cite[Lemma 2.2(3)]{BG2}, if the GK dimension of $M \mathcal S^{-1}$ (over $D$) is $0$, then $MS^{-1}$ is finite dimensional as $D$-space.
In view of Proposition \ref{GKdim_ch} it thus suffices to show that $M \mathcal S^{-1}$ is $D \s \langle x_{j_l} \rangle$-torsion for all $l \in \{1, \cdots, n - r \}$.
We set $A_l = \langle x_{i_1}, x_{i_2}, \cdots, x_{i_r}, x_{j_l} \rangle$.
In view of Proposition \ref{GKdim_ch}, $M$ is $F \s A_l$-torsion. 
Thus if $m \in M$ then $m\alpha_l = 0$ for some nonzero $\alpha_l \in F \s A_l$. If $\phi: M  \rightarrow M \mathcal  S^{-1}$ is the cannonical homomorphism then 
each element of $M \mathcal S^{-1}$ may be expressed as $\phi(m)s^{-1}$ for some $m \in M$ and $s \in S$. But then 
\begin{align*} 
\phi(m)s^{-1}(s\alpha_l) &= \phi(m)(s^{-1}s\alpha_l) \\ 
                           &= \phi(m)\alpha_l \\ 
			   &=  \phi(m\alpha_l) \\
			   &=  \phi(0) \\
                           &= 0.  
\end{align*} 
Now $s\alpha_l \in F \s A_l \subset D \s \langle x_{J_l} \rangle$. 
Hence $M\mathcal S^{-1}$ is $D \s \langle x_{J_l} \rangle$-torsion.   
The lemma now follows.
\end{proof}

We next note a useful property of the GK dimension for $D \s A$-modules was shown in \cite[
Lemma 2.2]{BG2}. 

\begin{proposition}[Brookes and Groves]\label{GKdim_prop}
Let $M$ be a finitely generated $D \s A$-module. If $B \le A$ is a subgroup and $N$ is a finitely generated $D \s B$-submodule of $M$, then 
$\gk(N) \le \gk(M)$. Furthermore, if $0 \rightarrow M' \rightarrow M \rightarrow M'' \rightarrow 0$ is an exact sequence of $D \s A$-modules, then 
\[ \gk(M) = \sup\{ \gk(M'), \gk(M'') \}. \]  
\end{proposition}

The GK dimension of a $D \s A$-module does not change in passing to a subgroup of finite index.  

\begin{proposition}\label{fi_sg}
Let $M$ be a finitely generated $D \s A$-module with $\gk(M) = r$, where $0 \le r \le \rk(A)$. 
If $A' < A$ is a subgroup of finite index in $A$ and $M'$ is the $D \s A'$-module structure on $M$, then 
$\gk(M') = r$.   
\end{proposition}

\begin{proof}
By Proposition \ref{GKdim_ch}, $A$ has a subgroup $B$ with rank $r$ so that $M$ is not $D \s B$-torsion. Set $B' = A' \cap B$. 
By the hypothesis in the proposition,
$[A :                                                                                                                                                                                                                                                                                                                                                                                                                                                                           A'] < \infty$ and hence $\rk(B') = r$.
Clearly $M'$ is not torsion as $D \s B'$-module. 
Hence $\gk(M') \ge r$ by Proposition \ref{GKdim_ch}.
Moreover $\gk(M') \le r$ by Proposition \ref{GKdim_prop}.
\end{proof}

In considering $D \s A$-modules, it is often useful to pass to a \emph{critical} submodule.
Critical $D \s A$ modules were introduced in \cite{BG2} and 
certain facts concerning them were also established (see \cite[Section 2]{BG2}) 
that account for their usefulness.  

\begin{definition}\label{def_cri_mod}
A nonzero $D \s A$-module $M$ is \emph{critcal} if for any nonzero proper submodule $N$, $\gk(M) > \gk(M/N)$. 
\end{definition}

\begin{proposition}[Proposition 2.5 of \cite{BG2}]\label{crit_mod}
Every nonzero $D \s A$-module contains a critical submodule. 
\end{proposition}

\begin{proposition}\label{holcri_is_sim}
A critical $D \s A$-module of minimum possible GK dimension is simple
\end{proposition}

\begin{proof}
Let $N$ be such a module and let $L$ be a nonzero proper submodule of $N$. Then 
$\gk(N/L) < \gk(N)$. However, this contradicts the definition of a critical module.  
\end{proof}

\begin{lemma}\label{loc_GKdim}
Let $N$ be a simple $F \s A$-module which is not torsion as 
$\mathcal C := F \s \langle x_1, \cdots, x_r \rangle$-module, where $\{ x_1, \cdots, x_n \}$ 
denotes a basis of $A$ and $0 < r < n$. 
Let   
\[  F(x_1, \cdots, x_r)[x_{r +1}, \cdots, x_n] \]
denote the right Ore localization of $F \s A$ at $\mathcal C \setminus \{0\}$ and let $N_r$ stand for the
corresponding localization of $N$.
Then 
\[ \gk(N_r) = \gk(N) - r, \] 
where the GK dimension of $N_r$ is measured relative to the division ring $F(x_1, \cdots, x_r)$.    
\end{lemma}
\begin{proof}
We note that $N$ must be a torsion-free $\mathcal C$-module. Indeed since $N$ is simple by the hypothesis 
the $C$-torsion submodule of $N$ is either $N$ or $0$. But the former possibility is ruled out since $N$ is not a 
torsion $C$-module by the hyoptehsis.      

Now $N$ satisfies the hypothesis of \cite[Lemma 4.5(ii)]{BG3} which asserts that $N_r$ is critical. 
However, in the proof it is actually shown that $N_r$ is critical with $\gk(N_r) = \gk(N) - r$.

\end{proof}

\section{The Brookes--Groves geometric invariant} 
\label{BG_inv}

We now continue our discussion of the module theory of crossed products $D \s A$ by describing the geometric invariant of Brookes and Groves that can be associated 
with each finitely generated $D \s A$-module. This invariant was modelled on the original Bieri--Strebel invariant which was used to  give a geometric criterion for a 
meta-abelian group to be finitely presented.  
Using this invariant we shall establish certain facts concerning $D \s A$-modules which will be used in our investigations that follow.

Let $M$ be a finitely generated $D \s A$-module.
The dual $A^* : = \hm_{\mathbb Z}(A, \mathbb R)$ of $A$ is easily seen to be a real vector space of dimension equal to the rank of $A$.
Hence we may identify $A^*$ with $\mathbb R^n$ where $n = \rk(A)$ and speak of group characters $\phi \in A^*$ as points. 
A point $\phi \in A^*$ determines a submonoid $A(0, \phi)$ and a subsemigroup $A(+, \phi)$ of $A$ as follows:
\begin{align*}\label{def_A0A+}
A(0,\phi)  &:= \{ a \in A \mid \phi(a) \ge 0 \} \\ 
A(+, \phi) &:= \{ a \in A \mid \phi(a) > 0 \}.
\end{align*}
For a subset  $X$ of $A$, we shall denote the subset of $F \s A$ of all elements $\alpha$ with $\supp(\alpha) \in X$ by $F \s X$.    
With each point $\phi \in A^*$ a module for $F \s \ker \phi$ may be associated as follows.
 
\begin{definition}\label{TC_mod}
Let $M$ be a finitely generated $D \s A$-module and let $\mathcal X$ be a (finite) generating set for $M$. 
For a point $\phi \in A^*$, the trailing coefficient module $TC_\phi(M)$ of $M$ at $\phi$ is defined as 
\[ TC_\phi(M) = \mathcal X(D \s A(0, \phi))/\mathcal X(D \s A(+, \phi)). \]      
\end{definition}  

It is easily seen that $TC_\phi(M)$ is a finitely generated $D \s \ker \phi$-module. In general, $TC_\phi(M)$ may depend on the choice of a generating set $\mathcal X$ 
for $M$. However, the next definition turns out to be independent of such a choice.

\begin{definition}\label{def_del_set}
Let $M$ be finitely generated $D \s A$-module. Then $\Delta(M)$ is defined to as the subset of all $\phi \in A^*$ so that 
$TC_\phi(M) = 0$.
\end{definition}

\begin{remark}
Since $A(0,0) = A$ and $A(+, 0)$ is empty, hence $TC_0(M) \ne 0$ if $M$ is nonzero. It follows that $0 \in \Delta(M)$ for nonzero $M$.   
\end{remark}

As already noted, $A^* \cong \mathbb R^n$ and so $\Delta(M)$ can be identified with a subset of $\mathbb R^n$.
A \emph{convex polyhedron} in $\mathbb R^n$ is an intersecion of a finite number of closed linear half spaces in $\mathbb R^n$.
A \emph{polyhedron} is a union of finitely many convex polyhedra. Suppose that a basis is fixed in $A$. A subspace of $A^*$ is \emph{rationally defined} if 
it has a set of generators each of which is a rational linear combination of the elements of the dual basis in $A^*$. 
A convex polyhedron is rational if it can be defined using half spaces with a rational subspace as boundary.
A polyhedron is rational if it is a finite union of rational convex polyhedra.
The dimension of a convex polyhedron is the dimension of the subspace spanned by it.
The dimension of a polyhedron is the maximum of the dimensions of its consitiuent convex polyhedra.

It was shown in \cite{BG1} that for a special class of $D \s A$-modules $\Delta(M)$ is a closed rational polyhedron and a weak 
polyhedrality result was obtained in \cite{BG2} for arbitary finitely generated $D \s A$-modules. Later it was shown in \cite{Wa1} that 
$\Delta(M)$ is a closed rational polyhedral cone for all finitely generated $D \s A$-modules $M$.

Given a subgrop $B \le A$, the inclusion mapping $\iota: B \hookrightarrow A$ induces the restriction mapping 
$\res_B: A^* \rightarrow B^*$. It is easily seen that the 
rational subspaces of $A^*$ are precisely the kernels of the maps $\res_B$ where $B$ ranges over all subgroups of $A$.  

For a subset $S$ of $\mathbb R^n$ and a point $x \in S$, a \emph{neighborhood} of $x$ in $S$ is the intersection with $S$ of a ball in 
$\mathbb R^n$ centered on $x$. The intersection of a ball in $R^n$ with an $m$-dimensional subspace will be called an $m$-ball.
A point of $x$ of $S$ wil be called \emph{regular} if some neighborhood of $x$ in $S$ is an $m$-ball and $S$ has no points with this property for a larger choice of $m$.             
For any subset $S$ of $\mathbb R^n$, the \emph{essential part} of $S^*$ is the (Eulclidean) closure of the set of the regular points of $S$.  
With this notation, Theorem 4.4 of \cite{BG2} asserts that for a finitely generated $D \s A$-module $M$, 
$\Delta^*(M)$ is a closed rational polyhedron of dimension equal to the GK dimension of $M$. The vector spans of 
nieghborhoods of the regular points in $\Delta^*(M)$ are the \emph{carrier spaces} of $\Delta^*(M)$. Note that each carrier space has dimension equal to $m$, where $m = \gk(M)$. 
A carrier space $V$ of $\Delta^*(M)$ is rational (see \cite{Gr2001}) and $V = \ker \res_B$ for some isolated subgroup $B \le A$, that is a subgroup $B \le A$ such that $A/B$ is torsion-free.
We also note that in this case \[ \dim(V) + \rk(B) = \rk(A), \] whence \[ m + \rk(B) = \rk(A).  \]
 
Now we shall employ the geometric invariant to obtain some results concerning $D \s A$-modules that we shall need later.

The following theorem was established in a special case in \cite{GU}. 
The reasoning given here is due to \cite{BG5}.

\begin{theorem}\label{com_mon}
Let $M$ be a finitely generated $F \s A$-module with $\gk(M) = r$, where $0 < r \le \rk(A) - 1$. Let 
$\{ x_1, \cdots, x_n \}$ be a basis of $A$ and suppose that $M$ is not torsion as 
$F \s \langle x_1, \cdots, x_r \rangle$-module. 
Then there exist $l_{r+ 1}, \cdots, l_n \in \langle x_1, \cdots, x_r \rangle$ and nonzero integers $s_{r+1}, \cdots, s_n$ so that 
\[ [\bar {l}_j\bar {x}_j^{s_j} , \bar {l}_k\bar {x}_k^{s_k}] = 1 \ \ \ \ \ \forall j,k \in \{ r+1, \cdots, n \}.  \]      
\end{theorem}

\begin{remark}
\label{eq_hyp}
The hypothesis in this theorem is equivalent to the localization 
\[ F(x_1, \cdots, x_r)[x_{r + 1}, \cdots x_n] \] 
having a nonzero finite dimensional module over the division ring $F(x_1, \cdots, x_r)$ (See \cite[Section 2]{Br2}).   
\end{remark}
\begin{proof}
Set $B = \langle x_1, \cdots, x_r \rangle$. It suffices to show that there exists a subgroup $E < A$ such that 
$B \cap E = \langle 1 \rangle$ and $F \s E$ is commutative. 
By the hypothesis the $F \s B$-torsion submodule $T_B(M)$ is a proper submodule of $M$. Moreover $N: = M/T_B(M)$ is 
$F \s B$-torsion-free. Any nonzero submodule of $N$ will also be torsion-free over $F \s B$ and so we may assume noting Proposition \ref{crit_mod} that $N$ is critical.

By \cite[Theorem 5.5]{BG2} and $\pi_B \Delta(N) = \Delta(L)$ for some cyclic critical $F \s B$-submodule $L$ of minimum dimension. Since $L$ is $F \s B$-torsion-free $\pi_B \Delta(N) = B^*$. Thus $\pi_B (\Delta^*(N)) = B^*$. Since the real vector space $B^*$ cannot be a union of a finite many proper subspaces, hence at least one of the $r$-dimensional spaces spanned by the convex polyhedra in $\Delta^*(N)$ is mapped surjectively onto $B^*$. Let $V$ be such a subsepace.
Since $C^*$ is also $r$-dimensional, hence 
\begin{equation}\label{proj}
\ker \pi_B +  V  = \ker \pi_B \oplus V = A^*.
\end{equation}
Since $V$ is a rational subspace, $V = \ker \pi_C$ for some $C \le A$ such that $C$. 
Moreover $C$ has a subgroup $E$ of finite index such that $F \s E$ is commutative. It easily follows from (\ref{proj}) that $B \cap E = \langle 1 \rangle$.
\end{proof}

\begin{lemma}\label{css_p}
Let $M$ be a nonzero finitely generated $D \s A$-module with $\gk(M) \ge 1$. 
Let $C < A$ be an infinite cyclic subgroup of $A$ such that 
$M$ is $D \s C$-torsion. Let $V$ be a carrier space of $\Delta^*(M)$. Then $V^\circ \cap C \ne \langle 1 \rangle$. 
\end{lemma}

\begin{proof}
Set $B_V = V^\circ$. Suppose that $B_V \cap C  = \langle 1 \rangle$. It is known (e.g., \cite[Section 1]{Gr2001}) that the convex polyhedral cone in $\Delta^*(M)$ 
which spans $V$ contains a point  $\psi$ so that 
$\ker \psi = B_V$. Since $\psi$ is a point of $\Delta(N)$, $TC_\psi(N) \ne 0$. Let $\mathcal X$ be a finite generating set for $N$. 
It is easily seen from the definition of $TC_\psi(N)$ that there exists $x \in \mathcal X$ so that 
\begin{equation}\label{TCM_nz}
x \not \in \mathcal X(D \s A(+)).
\end{equation}
By the hypothesis in the lemma $M$ is $D \s C$-torsion and so there exists 
$\gamma \in F \s C \setminus \{0\}$ so that $x\gamma = 0$. Moreover if $c_i, c_j \in \supp(\gamma)$ are distinct elements then 
$\psi(c_i) \ne \psi(c_j)$. For otherwise, \[ c_ic_j^{-1} \in C \cap \ker \psi = C \cap B_V \] 
contrary to our assumption that $B_V \cap C  = \langle 1 \rangle$. Hence multiplying by a suitable unit in $F \s C$ we may assume that $\gamma = 1 - \gamma_\psi$, where 
$\gamma_\psi \in D \s C \cap D \s A(+)$. Tnen we have $x = x\gamma_\psi \in \mathcal X(D \s A(+))$ contray to (\ref{TCM_nz}).     

Hence $B_V \cap C  \ne \langle 1 \rangle$. 
\end{proof}

\section{The dimension of a tensor product}
The tensor product $F \s A_1 \otimes_F F \s A_2$ of twisted group algebras is a twisted group algebra of $A_1 \times A_2$ over $F$.
We shall write $\otimes$ for $\otimes_F$ when there is no danger of confusion. 
We begin with the following characterization of $\dim(F \s A)$ which was conjectured in \cite[Section 3.3]{MP} and was shown in \cite[Theorem A]{Br2}.

\begin{theorem}[C.J.B. Brookes]\label{dim_sup}
The dimension of an algebra $F \s A$ equals the supremum of the ranks of subgroups $B \le A$ so that $F \s B$ is commutative.
\end{theorem} 

We note a few consequences of the above theorem.

\begin{corollary}\label{dim_fi_sg}
Given an algebra $F \s A$,
let $A' < A$ be a subgroup of finite index in $A$.
Then \[ \dim(F \s A') = \dim(F \s A). \]  
\end{corollary}

\begin{corollary}\label{ten_dim}
Let $F \s A_1$ and $F \s A_2$ be arbitrary twisted group algebras. 
Then 
\begin{equation}\label{ten_ineq}
 \dim(F \s A_1 \otimes_F F \s A_2 ) \ge \dim(F \s A_1) + \dim(F \s A_2).  
\end{equation}
\end{corollary}

In Theorem \ref{ub_gen} we shall show an upper bound for $\dim(F \s A_1 \otimes_F F \s A_2 )$.  
We shall need the next few facts in the proof of this theorem.
The following fundamental result was shown in \cite[Theorem 3]{Br2}.

\begin{theorem}[Brookes]\label{F-ab_sgrp}  
If an algebra $F \s A$ has a finitely generated nonzero module $N$ with $\gk(N)= m$, then $A$ contains a subgroup $B$ with rank $\rk(A) - m$ such that 
$F \s B$ is commutative.    
\end{theorem}

\begin{corollary}\label{min_dim}
Let $M$ be a nonzero finitely generated module over an algebra $F \s A$. Then 
$\gk(M) \ge \rk(A) - \dim(F \s A)$. 
\end{corollary}

\begin{proof}
 Indeed if this were not true then by the last theorem, 
$A$ would contain a subgroup $B$ with $\rk(B) > m$ so that $F \s B$ is commutative. But this contradicts Theorem \ref{dim_sup}.  
\end{proof}

We shall use the last result in the following form. 

\begin{lemma}\label{min_mod}
If an algebra $F \s A$ has dimension $m$ then it has a  
simple module $N$ with $\gk(N) = \rk(A) - m$.  
\end{lemma}

\begin{proof}
It was shown in \cite[Section 2]{Br2} that there exists a nonzero finitely generated $F \s A$-module $M$ with $\gk(M) = \rk(A) - m$.
Here we merely argue that such a module $M$ has a simple submodule whose GK dimension must be $\rk(A) - m$ in view of Corollary \ref{min_dim}.
But in view of \cite[Lemma 5.6]{MP} and Corollary \ref{min_dim}, $M$ has finite length.
\end{proof}

We now prove our first main result.

\begin{theorem}\label{ub_gen}
Given algebras $F \s A_1$ and $F \s A_2$ suppose that $\dim(F \s A_i) < \rk(A_i)$ for $i \in \{1,2\}$. Let 
$d := \dim(F \s A_1 \otimes_F F \s A_2)$
\begin{equation}
\label{dim_bd}
 d \le \min\{ \dim(F \s A_1) + \rk(A_2), \dim(F \s A_2) + \rk(A_1)   \} - 1. 
\end{equation}  
\end{theorem}

\begin{proof}

Without loss of generality, we may assume that 
\[ \dim(F \s A_1) + \rk(A_2)  \le \dim(F \s A_2) + \rk(A_1) . \] 
We will show that
\[ \dim(F \s A_1 \otimes_F F \s A_2) \le \dim(F \s A_1) + \rk(A_2) - 1. \]   
Set $r_i  = \rk(A_i)$,  $d_i = \dim(F \s A_i)$ for $i \in \{1,2\}$, $d  = \dim(F \s A_1 \otimes_F F \s A_2)$ and $\mathcal A = F \s A_1 \otimes_F F \s A_2$.  

If $d > d_1 + r_2$, then by Lemma \ref{min_mod}, $\mathcal A$ has a nonzero finitely generated module $M$ with
$\gk(M) < r_1 - d_1$. Let $M_1$ be a nonzero $F \s A_1$-submodule of $M$. 
Then $\gk(M_1) \le \gk(M)$ by Proposition \ref{GKdim_prop}. But this is impossible in view of Corollary \ref{min_dim}.

We now suppose that $d = d_1 + r_2$. In this case, by Lemma \ref{min_mod}, $\mathcal A$ has a simple module $M'$ with $\gk(M') = r_1 - d_1$. Let $M_1'$ be a finitely generated $F \s A_1$-submodule of $M$. 
In view of  Corollary \ref{min_dim} and Proposition \ref{GKdim_prop} we have
\[ r_1 - d_1 \le \gk(M_1') \le \gk(M') = r_1 - d_1, \]
whence $\gk(M_1') = r_1 - d_1$.  
Let $\{ x_1, \cdots, x_{r_1} \}$ be a basis in $A_1$. 
By Proposition \ref{GKdim_ch}, there must be a subset 
$\{j_1, \cdots, j_{r_1 - d_1}\} \subset \{1, \cdots, r_1\}$ 
so that $M_1'$ and hence $M'$ is not $F \s \langle x_{j_1}, \cdots, x_{j_{r_1 - d_1}} \rangle$-torsion.
There is no harm in  assuming that $(j_1, \cdots, j_{r_1 - d_1}) = (1, \cdots, r_1 - d_1)$.
By Lemma \ref{fin_D_dim}, the crossed product  
\[ F(x_1, \cdots, x_{r_1 - d_1})[x_{r_1 - d_1 + 1}, \cdots, x_{r_1}, y_1, \cdots, y_{r_2}] \] 
has a nonzero module finite dimensional as $F(x_1, \cdots, x_{r_1 - d_1})$-space. 
Moreover by Theorem \ref{com_mon}, 
there exist 
\[ l_{i}, l'_{j} \in \langle x_1, \cdots, x_{r_1 - d_1} \rangle  \ \mbox{and} \ s_i, s'_j \in \mathbb Z - \{0\}  \   \  \  \  \ \forall i \in \{r_1 - d_1 +1, \cdots, r_1 \}, j \in \{1, \cdots, r_2 \} \] 
such that 
\begin{align}\label{com_1}
[ \bar {l}_i\bar {x}_i^{s_i}, \bar {l}_k\bar {x}_k^{s_k} ] &= 1  \  \  \  \  \ \forall i,k  \in \{r_1 - d_1 +1, \cdots, r_1 \},
\\ \label{com_2}
[ \bar {l}_i\bar {x}_i^{s_i}, \bar {l}_j'\bar {y}_j^{s'_j} ] &= 1  \  \  \  \  \ \forall i  \in \{r_1 - d_1 +1, \cdots, r_1 \}, j \in \{1, \cdots, r_2 \}, \ \ \mbox{and}
\\ \label{com_3}
[ \bar {l}_j'\bar {y}_j^{s'_j}, \bar {l}_p'\bar {y}_p^{s'_p} ] &= 1  \  \  \  \  \ \forall j,p \in \{1, \cdots, r_2 \}. 
\end{align}  
 
We note that (\ref{com_1}) means that for 
\[ C_1 : = \langle  {l}_{r_1 - d_1 + 1} {x}_{r_1 - d_1 + 1}^{s_{r_1 - d_1 + 1}}, \cdots,  {l}_{r_1} {x}_{r_1}^{s_{r_1}} \rangle \] 
the subalgebra $F \s C_1$ is commutative. 
Since $\mathcal A$ is a tensor product of $F \s A_1$ and $F \s A_2$, we have 
\begin{equation}\label{com_4} 
[ \bar {l}_i\bar {x}_i^{s_i}, \bar {y}_j^{s'_j} ] = 1  \  \  \  \  \ \forall i  \in \{r_1 - d_1 +1, \cdots, r_1 \}, j \in \{1, \cdots, r_2 \}. 
\end{equation}   
Applying (\ref{com_4}) and (\ref{com_id_2}) to (\ref{com_2}) we get 
\begin{align}\label{com_5}
1 &= [ \bar {l}_i\bar {x}_i^{s_i}, \bar {l}_j'\bar {y}_j^{s'_j} ]  \\ \label{com_6}  
  &= [\bar {l}_i\bar {x}_i^{s_i}, \bar {l}_j'][ \bar {l}_i\bar {x}_i^{s_i}, \bar {y}_j^{s'_j} ] \\ \label{com_7}
  &= [\bar {l}_i\bar {x}_i^{s_i}, \bar {l}_j']  \ \ \ \ \  \forall i  \in \{r_1 - d_1 +1, \cdots, r_1 \}, j \in \{1, \cdots, r_2 \}. 
\end{align}
By Theorem \ref{dim_sup}, the subgroup $C_1$ of $A_1$ 
has the maximal possible rank with respect to $F \s C_1$ being commutative.
But then (\ref{com_5}) -- (\ref{com_7}) imply that 
\[ {l}_j'^{t_j} \in C_1   \ \ \ \ \  \forall j \in \{1, \cdots, r_2 \}, \]  
where $t_j$ is a positive integer for each $j$.  
We now use the fact that $\mathcal A$ is a tensor product and apply (\ref{com_id_1})--(\ref{com_id_2}) to (\ref{com_3}). We get
\begin{align*}
1 &=  [ \bar {l}_j'\bar {y}_j^{s'_j}, \bar {l}_p'\bar {y}_p^{s'_p} ] \\
  &=  [ \bar {l}_j', \bar {l}_p'][ \bar {y}_j^{s'_j}, \bar {y}_p^{s'_p} ] \\
  &=  \bigl ( [ \bar {l}_j', \bar {l}_p'][ \bar {y}_j^{s'_j}, \bar {y}_p^{s'_p} ] \bigr )^{t_jt_p}  \\ 
  &=  [ \bar {l}_j', \bar {l}_p']^{t_jt_p}[ \bar {y}_j^{s'_j}, \bar {y}_p^{s'_p} ]^{t_jt_p} \\
  &=  [\bar {l}_j'^{t_j}, \bar {l}_p'^{t_p}][ \bar {y}_j^{s'_jt_j}, \bar {y}_p^{s'_pt_p} ] \\
  &=  [ \bar {y}_j^{s'_jt_j}, \bar {y}_p^{s'_pt_p} ]  \ \ \ \ \   \forall j,p \in \{1, \cdots, r_2 \}, 
\end{align*} 
where in the last step we have made use of the commutativity of $F \s C$. 
We have just shown that 
\[ [ \bar {y}_j^{s'_jt_j}, \bar {y}_p^{s'_pt_p} ] = 1  \ \ \ \ \   \forall j,p \in \{1, \cdots, r_2 \},  \]
where $s'_j$ and $t_j$ are nonzero integers for all $j$.
It follows that $A_2$ has a subgroup $A_2'$ of finite index so that $F \s A_2'$ is commutative.
By Theorem \ref{dim_sup}, $d_2 = r_2$ contrary to the hypothesis.
\end{proof}

\begin{remark}
If we drop the assumption $\dim(F \s A_i) < \rk(A_i)$ from the preceding theorem, the theorem may no longer hold true. 
For example, suppose that $\rk(A_i) = 2$ for $i \in \{1,2\}$, $\dim(F \s A_1) = 2$ and $\dim(F \s A_2) = 1$. By Corollary \ref{ten_dim}, $\dim( F \s A_1 \otimes_F F \s A_2) \ge 3$. But 
\[  \min\{ \dim(F \s A_1) + \rk(A_2) - 1, \dim(F \s A_2) + \rk(A_1) - 1  \} = \min\{3, 2\} = 2. \]
\end{remark}

\begin{remark}\label{codim0}
The hypotheses $\dim(F \s A_i) < \rk(A_i)$ is used only in the second part of the proof to show that $d \ne d_1 + r_2$.
The fact \[ \dim(F \s A_1 \otimes F \s A_2) \le \min\{\dim(F \s A_1) + \rk(A_2), \dim(F \s A_2) + \rk(A_1) \} \]
remains valid in the case $\dim(F \s A_i) = \rk(A_i)$ for some $i \in \{1,2\}$. 
\end{remark}
The next corollary shows that if atleast one of the algebras $F \s A_i$ is ``virtually commutative'' then the dimension is additive.

\begin{corollary}\label{codim}
Given twisted group algebras $F \s A_1$ and $F \s A_2$, 
if $\dim(F \s A_i) = \rk(A_i)$ for some $i \in \{1,2 \}$ then equality holds in  (\ref{ten_dim}).
\end{corollary}

\begin{proof}
Without loss of generality, we may assume that $\dim(F \s A_i) = \rk(A_i)$. 
As in Remark \ref{codim0}, 
\begin{align*}
\dim(F \s A_1 \otimes F \s A_2) &\le \min\{\dim(F \s A_1) + \rk(A_2), \dim(F \s A_2) + \rk(A_1) \} \\ 
                                &= \min\{\rk(A_1) + \rk(A_2), \dim(F \s A_2) + \dim(F \s A_1) \} \\                     
                                &= \dim(F \s A_1) + \dim(F \s A_2).
\end{align*}
Hence equality holds in  (\ref{ten_ineq}).   
\end{proof}

Equality holds in (\ref{ten_ineq}) in the following situation also. 

\begin{corollary}\label{eq_case1}
Let $F \s A_1$ and $F \s A_2$ be twisted group algebras such that 
$\dim(F \s A_i) < \rk(A_i)$. If $\dim(F \s A_i) = \rk(A_i) - 1$ for some $i \in \{1,2\}$, then 
\[ \dim( F \s A_1 \otimes_F F \s A_2)  =  \dim(F \s A_1) + \dim(F \s A_2). \]
\end{corollary}

\begin{proof}
Without loss of generality we assume that $\dim(F \s A_1) = \rk(A_1) - 1$. 
We set $r_i  = \rk(A_i)$,  $d_i = \dim(F \s A_i)$ for $i \in \{1,2\}$ and $d  = \dim(F \s A_1 \otimes_F F \s A_2)$.
By Theorem \ref{ub_gen}, 
\begin{align}
d &\le \min\{ d_1 + r_2, d_2 + r_1 \} -1  \\ 
  &= \min\{r_1 + r_2 - 1, d_2 + r_1\} - 1 \\
  &=  d_2 + r_1 - 1\\ 
  &= d_2 + d_1. 
\end{align}
On the other hand $d \ge d_1 + d_2$ by Corollary \ref{ten_dim}. Hence $d = d_1 + d_2$.

We recall our definition of $\codim(F \s A)$ given in Section \ref{intr} as \[ \codim(F \s A) = \rk(A) - \dim(F \s A). \]
With this the last two corollaries may be combined in the following form.

\begin{corollary}\label{ten_eq}
Given algebras $F \s A_1$ and $F \s A_2$ suppose that $\codim(F \s A_i) \le 1$ holds for some $i \in \{1,2\}$.
Then 
\[ \dim(F \s A_1 \otimes_F F \s A_2 ) = \dim(F \s A_1) + \dim(F \s A_2). \]
\end{corollary}

We recall that an algebra $\otimes_{i  = 1}^k F \s B_i$, where 
$B_i \cong \mathbb Z^2$ is known as the multiplicative analouge of the Weyl algebra~\cite{Ja}. As $1 \le \dim(F \s B_i) \le 2$ (Section \ref{intr}) we obtain the following.
\begin{corollary}
For any algebras $F \s B_j$, where $B_j \cong \mathbb Z^2$ and $j \in \{1, \cdots, s\}$,
\[ \dim(F \s B_1 \otimes F \s B_2 \otimes \cdots \otimes F \s B_s) = \sum_{j = 1}^s\dim(F \s B_j). \]  
 \end{corollary}


\subsection{The inequality becomes strict}

As already noted in the example in Section \ref{intr}, the upper bound in Theorem \ref{ub_gen} is attained for the tensor products in which 
each tensor factor has dimension one. It easily follows from Corollary \ref{ten_eq} that this upper bound is also attained in the case when both the (tensor) factors have codimension one.  

We shall now prove (see Theorem \ref{dim_<_ub} ) that in the case both the dimension and the codimension of each (tensor) factor exceeds one 
the inequality in Theorem \ref{ub_gen} is strict.

\begin{theorem}\label{dim_<_ub}
Let $F \s A_1$ and $F \s A_2$ be twisted group algebras such that
$\codim(F \s A_i) \ge 2$, $\dim(F \s A_i) \ge 2$ and $F \s A_i$ has center $F$.
Let $d : = \dim(F \s A_1 \otimes F \s A_2)$.
Then  
\begin{equation}\label{ub}
d < \min\{\dim(F \s A_1) + \rk(A_2), \dim(F \s A_2) + \rk(A_1) \} - 1. 
\end{equation}

\end{theorem}
  
\begin{proof}
Set $\mathcal A  =  F \s A_1 \otimes F \s A_2, n_i = \rk(A_i)$ and $d_i = \dim(F \s A_i)$, where $i \in \{1, 2\}$.
Let $ \{x_1, \cdots, x_{n_1} \}$ and $\{ y_1, \cdots, y_{n_2} \}$ be bases in $A_1$ and $A_2$ respectively.

Without loss of generality, we may assume that 
\[ \codim(F \s A_1) \ge \codim(F \s A_2). \] 
In other words,
$n_1 - d_1 \ge n_2 - d_2$ and so $n_1 + d_2 \ge n_2 + d_1$. 
Consequently the right side of (\ref{ub}) is 
$n_2 + d_1 - 1$. 
We must thus show that 
\begin{equation}\label{dim_<_inf}
 \dim(\mathcal A) < n_2 + d_1 - 1. 
\end{equation} 
But in view of Theorem \ref{ub_gen}, it suffices to show that   
\[ \dim(\mathcal A) \ne n_2 + d_1 - 1. \]  
To this end we suppose that 
\[ \dim(\mathcal A) = n_2 + d_1 - 1. \] 
 Applying Lemma \ref{min_mod}, $\mathcal A$ has a simple module $N$ with 
\[ \gk(N) = n_1 - d_1 + 1. \]
Using Proposition \ref{crit_mod},  let $N_1$ be a finitely generated critical $F \s A_1$-submodule of $N$ of the maximum possible GK dimension.
In view of Corollary \ref{min_dim} and 
Proposition \ref{GKdim_prop}, 
\begin{equation}\label{dim_val} 
n_1 - d_1 \le \gk(N_1) \le  n_1 - d_1 + 1. 
\end{equation}

We first assume  that 
\begin{equation}\label{low_dim_val}
\gk(N_1) = n_1 - d_1.
\end{equation} 
By Proposition \ref{GKdim_ch} and Remark \ref{loc_rmk}, $N_1$ and hence $N$ is not $F \s \langle x_1, x_2, \cdots, x_{n_1 - d_1} \rangle$-torsion. 
This means that we may localize $n_1 - d_1$ generators  say $x_1, \cdots, x_{n_1 - d_1}$ of $A_1$. 
We then obtain the crossed product 
\[ F(x_1, \cdots, x_{n_1 - d_1})[x_{n_1 - d_1 + 1}, \cdots, x_{n_1}, y_1, \cdots, y_{n_2}]. \]
Set $\hat {A}_2 = \langle x_{n_1 - d_1 + 1}, \cdots, x_{n_1}, y_1, \cdots, y_{n_2} \rangle$ and note that $\rk(\hat{A}_2) > \rk(A_2)$.   
By Lemma \ref{loc_GKdim},  the corresponding localization 
$N_{n_1 - d_1}$ of $N$ has GK dimension $1$ (over $F(x_1, \cdots, x_{n_1 - d_1})$) since
 \[ \gk(N) = n_1 - d_1 + 1. \]

Let $V$ be a carrier space of $\Delta^*(N_{n_1 - d_1})$. Then 
\[ \dim(V) = \gk(N_{n_1 - d_1}) = 1. \] 
Let $B_V$ be the isolated subgroup of $\hat A_2$ so that $V = B_V^\circ$.  Then $\rk(B_V) = \rk(\hat {A}_2) - 1$.

We note that $N_{n_1 - d_1}$ must be $F(x_1, \cdots, x_{n_1 - d_1})[x_j]$-torsion for each $j \in \{ n_1 - d_1 + 1, \cdots, n_1\}$ otherwise
$N$ would not be $F \s \langle x_1, \cdots, x_{n_1 - d_1}, x_j \rangle$-torsion. 
But this is a contradiction to (\ref{low_dim_val}) in view of Proposition \ref{GKdim_ch}.

Thus by Lemma \ref{css_p}, there exist nonzero integers $s_j$ so that 
$x_j^{s_j} \in B_V$ for all $j \in \{ n_1 - d_1 + 1, \cdots, n_1\}$.
Since $\rk(B_V) = \rk(\hat {A}_2) - 1$, hence 
\begin{align*} 
\rk(B_V \cap A_2) &\ge \rk(B_V) + \rk(A_2) - \rk(\hat {A}_2) \\ 
                  &= \rk(B_V) - \rk(\hat A_2) +  \rk(A_2) \\ 
                  &= \rk(A_2) - 1
\end{align*}
\end{proof}
But as noted above $x_j^{s_j} \in B_V$ where 
$j \in \{ n_1 - d_1 + 1, \cdots, n_1\}$ and so $\rk(B_V \cap A_1) = d_1$. Now  \[ 
\rk(B_V)\ge \rk(A_1 \cap B_V)+ \rk(A_2 \cap B_V) \ge  d_1 + \rk(A_2) - 1 = \rk(B_V) \]  
and hence $\rk(B_V \cap A_2) = \rk(A_2) - 1$. There will be no harm in assuming that $y_1, \cdots, y_{n_2 - 1}$  generate $B_V \cap A_2$ as our arguments remain valid in passing to a subgroup of $A_2$ of finite index (Theorem \ref{dim_sup}, Propositions \ref{c_a_fis} and Corollary \ref{dim_fi_sg}).
By \cite[Theorem 3]{Gr2001},
the crossed product 
\[ F(x_1, \cdots, x_{n_1 - d_1})[x_{n_1 - d_1 + 1}^{s_{n_1 - d_1 + 1}}, \cdots, x_{n_1}^{s_{n_1}}, y_1, \cdots, y_{n_2} ] \] 
has a nonzero module finite dimensional over $F(x_1, \cdots, x_{n_1 - d_1})$.
We can thus apply Theorem \ref{com_mon} noting remark \ref{eq_hyp}. 
Again  we may assume $s_i = 1$ in Theorem \ref{com_mon}.    
We thus obtain 
\begin{align}\label{cmt_3}
[\bar {a}_i\bar {x}_{i}^{s_i}, \bar {a}_j \bar {x}_{j}^{s_j}] &= 1  \ \ \ \ \ \forall i,j \in \{ n_1 - d_1 +1, \cdots, n_1 \},\\  \label{cmt_4} 
[\bar {a}_i \bar {x}_{i}^{s_i}, \bar {b}_j \bar {y}_j] &= 1  \ \ \ \ \ \forall i \in \{ n_1 - d_1 +1, \cdots, n_1 \}, j \in \{ 1, \cdots, n_2 - 1 \},\\ \label{cmt_5} 
[\bar {b}_i \bar{y}_i, \bar {b}_j \bar {y}_j  ] &= 1 \ \ \ \ \ \forall i,j \in \{ 1, \cdots, n_2 - 1 \}, 
\end{align}
where $a_i, b_j \in  \langle x_1, \cdots, x_{n_1 - d_1} \rangle$.
But (\ref{cmt_4}) yields:
\begin{equation}\label{cmt_6}
[a_ix_{i}^{s_i}, \bar {b}_j] = 1  \ \ \ \ \forall i \in \{ n_1 - d_1 +1, \cdots, n_1 \}, j \in \{ 1, \cdots, n_2 - 1 \}.
\end{equation}
 
By definition, $d_1 = \dim(F \s A_1)$ and so $\bar {b}_j  = 1$ for all $j$, otherwise  by (\ref{cmt_3}) and (\ref{cmt_6}), there is a subgroup $B_1 < A_1$ with 
$\rk(B_1) \ge  d_1 + 1$ so that $F \s B_1$ is commutative. 
But this is impossible by Theorem \ref{dim_sup}. 
Hence $b_j = 1$ for all $j \in \{ 1, \cdots, n_2 - 1. \}$.
It then follows from (\ref{cmt_5}) that 
\[ [ \bar{y}_i, \bar {y}_j  ] = 1 \ \ \ \ \ \forall i,j \in \{ 1, \cdots, n_2 - 1 \}. \] 
Hence $A_2$ has a subgroup $B_2$ with $\rk(B_2) = n_2 - 1$ so that 
$F \s B_2$ is commutative. This is contrary to the hypothesis in the theorem that 
$\dim(F \s A_2) < n - 1$.  

In view of (\ref{dim_val}), we may now assume that  $\gk(N_1) = n_1 - d_1  + 1$.
By Proposition \ref{GKdim_ch} (and the succeeding paragraph) this implies that the ring 
\[ F(x_1, \cdots, x_{n_1 - d_1 + 1})[x_{n_1 - d_1 + 2}, \cdots, x_{n_1}, y_1, \cdots, y_{n_2}] \] 
has a nonzero module finite dimensional over the division ring $F(x_1, \cdots, x_{n_1 - d_1 + 1})$. For the sake of convinience, we may assume that 
\begin{equation}\label{ins_comtr_1}
[\bar {y}_k, \bar {y}_l] = 1  \ \ \ \ \  \forall k,l \in \{ 1, \cdots, d_2 \} 
\end{equation}
noting Theorem \ref{dim_sup}, Propositions \ref{c_a_fis} and Corollary \ref{dim_fi_sg}.  
We may apply Theorem \ref{com_mon} noting Remark \ref{eq_hyp}.
We thus have the following relations.
\begin{align}\label{comt_3}
[\bar {a}_i\bar {x}_{i}, \bar {a}_j \bar {x}_{j}] &= 1  \ \ \ \ \ \forall i,j \in \{ n_1 - d_1 + 2, \cdots, n_1 \},\\  \label{cmt_4} 
[\bar {a}_i \bar {x}_{i}, \bar {b}_j \bar {y}_j] &= 1  \ \ \ \ \ \forall i \in \{ n_1 - d_1 + 2, \cdots, n_1 \}, j \in \{ 1, \cdots, n_2  \},\\ \label{cmt_5} 
[\bar {b}_i \bar{y}_i, \bar {b}_j \bar {y}_j  ] &= 1 \ \ \ \ \ \forall i,j \in \{ 1, \cdots, n_2 \}, 
\end{align}
where $a_i, b_j \in  \langle x_1, \cdots, x_{n_1 - d_1 + 1} \rangle$.
But as $\mathcal A$ is a tensor product, (\ref{ins_comtr_1}) and (\ref{cmt_5}) together mean that 
\begin{equation}\label{ins_cmtr_2}
[\bar {b}_k, \bar {b}_l] =  1 \ \ \ \ \ \ \forall k,l \in  \{ 1, \cdots, d_2 \}
\end{equation}
On the other hand (\ref{cmt_4}) yields:
\begin{equation}\label{comt_4}
[\bar {a}_i \bar {x}_{i}, \bar {b}_j] = 1  \ \ \ \ \ \forall i \in \{ n_1 - d_1 + 2, \cdots, n_1 \}, j \in \{ 1, \cdots, n_2  \}.
\end{equation}
Now the $\bar {a}_i \bar {x}_{i}$ are $d_1 - 1$ independent commuting monomials all of which commute with the $\bar {b}_j$. But as $\{\bar {b}_1, \cdots, \bar {b}_{d_2} \}$ commute mutually, we must have   
\[ \rk(\langle b_1, \cdots, b_{d_2} \rangle) = 1. \]
For otherwise we would obtain a contradiction to Theorem \ref{dim_sup} (applied to $F \s A_1$).

As $d_2 \ge 2$ by the hypothesis in the theorem, hence $b_1$ and $b_2$ are mutually dependent. We may thus find integers $t_1$ and $t_2$ 
so that $b_1^{t_1} = b_2^{t_2}$. Pick $j$ so that  $d_2 + 1 \le j \le n_2$.
We have from (\ref{cmt_5}):
\[ [b_1, b_j][y_1, y_j] = 1 \] 
and so \[ [b_1^{t_1},b_j][y_1^{t_1}, y_j] = 1. \] 
Similarly 
 \[ [b_2^{t_2},b_j][y_2^{t_2}, y_j] = 1. \] 
whence combining the last two equations we get
\[ [y_1^{t_1}y_2^{-t_2}, y_j] = 1 \] 
But this means that $F \s A_2$ has center larger than $F$ contrary to the hypothesis in the theorem.
This concludes our proof.
\end{proof}

\begin{corollary}\label{add_2}
Suppose that the algebras  $F \s A_1$ and $F \s A_2$ satisfy the following conditions
\begin{itemize}
\item[(i)] $\dim(F \s A_1),\dim(F \s A_2) \ge 2$,
\item[(ii)] $\codim(F \s A_1) \ge \codim(F \s A_2) = 2$, 
\item[(iii)] $F \s A_i$ has center $F$.  
\end{itemize}
Then \[ \dim(F \s A_1 \otimes F \s A_2) = \dim(F \s A_1) + \dim(F \s A_2) \] 
\end{corollary}

\begin{proof}
Let $T: = F \s A_1 \otimes F \s A_2$, $n_i = \rk(A_i)$ and $d_i = \dim(F \s A_i)$, where $i \in \{1,2\}$.   
Hypothesis (ii) in the theorem means that $n_1 - d_1 \ge n _2 - d_2$ and thus $n_2 + d_1 \le n_1 + d_2$.   
By Theorem \ref{dim_<_ub} we get   
\begin{align*}
\dim(T) &<  \min\{d_1 + n_2, d_2 + n_1 \} - 1 \\
        &=  \dim(F \s A_1) + \rk(A_2) - 1 \\
         &=  \dim(F \s A_1) + \dim(F \s A_2) + 1.
\end{align*}
The corollary now follows from Corollary \ref{ten_dim}. 
 \end{proof}

\end{document}